\documentclass[11pt,a4paper,leqno]{article}   % leqno = equation numbers on the left (common in geometry journals)

\usepackage[utf8]{inputenc}
\usepackage[T1]{fontenc}
\usepackage{amsmath,amssymb,amsthm}
\numberwithin{equation}{section} %number equations within sections
\usepackage{mathtools}        % fixes and extends amsmath
\usepackage{mathrsfs}         % \mathscr
\usepackage{bm}               % bold math
\usepackage{geometry}
\usepackage{hyperref}
\hypersetup{
    colorlinks=true,
    linkcolor=blue,
    citecolor=blue,
    urlcolor=blue,
    pdftitle={Spectral bounds}, % Title shown in PDF viewers
    pdfauthor={Fei He}                                          % Author metadata for the PDF
}
\usepackage{cleveref}         % \cref, \Cref
\usepackage{enumitem}         % custom lists
\usepackage{graphicx}         % if you need figures
\theoremstyle{plain}
\newtheorem{theorem}{Theorem}[section]   % Theorem numbers include section: Theorem 1.1, 2.1, etc.
\newtheorem{lemma}[theorem]{Lemma}       % Lemma shares counter with Theorem
\newtheorem{proposition}[theorem]{Proposition}
\newtheorem{corollary}[theorem]{Corollary}

\theoremstyle{definition}

\theoremstyle{remark}
\newtheorem{remark}[theorem]{Remark}

\DeclareMathOperator{\Rm}{Rm}
\DeclareMathOperator{\Ric}{Ric}
\DeclareMathOperator{\Sc}{S}
\DeclareMathOperator{\tr}{tr}

\DeclareMathOperator{\divg}{div}

\DeclareMathOperator{\Vol}{Vol}
\DeclareMathOperator{\End}{End}
\DeclareMathOperator{\rank}{rank}
\DeclareMathOperator{\dist}{dist}
\DeclareMathOperator{\supp}{supp}

\DeclareMathOperator{\Ker}{Ker}

\newcommand{\dd}{\mathrm{d}}   % upright d for exterior derivative (standard in geometry)
\newcommand{\D}{\nabla}
\newcommand{\pd}{\partial}      % quick way to type \partial

\title{
    \bfseries
    Spectral bounds for $f$-Laplace-type operators with applications to Betti numbers on gradient Ricci shrinkers\\[4pt]
}
\author{Fei He\thanks{Partially supported by Xiamen Municipal Natural Science Foundation (No. 3502Z202673005) and NSFC (No. 12141101).}}
\date{}
\newcommand{\contactinfo}{%
    \par\vspace{\baselineskip}%
    \noindent 
    School of Mathematical Science, Xiamen University, Xiamen, China.\\
    E-mail address:    \texttt{hefei@xmu.edu.cn}
}

\begin{document}

\maketitle

\begin{abstract}
We study the spectrum of the $f$-Laplacian on complete gradient Ricci shrinkers.  Upper and lower bounds for the $k$-th eigenvalue 
are established in terms of the volume growth rate.  Both bounds are shown to be sharp in the exponent. The method extends to $f$-Laplace-type operators on vector bundles; as an application we obtain explicit upper bounds for the Betti numbers.
\end{abstract}

% =====================================================
\section{Introduction}
\label{sec:intro}

A gradient Ricci shrinker is a complete Riemannian manifold $(M^n,g)$ equipped with a smooth potential function $f$ satisfying
\begin{equation}\label{eq:intro-shrinker}
\Ric+\nabla^2 f=\tfrac12 g .
\end{equation}
Ricci shrinkers model finite-time singularities of the Ricci flow and have been studied intensively over the past two decades.  By the fundamental work of
Cao--Zhou~\cite{CaoZhou}, $f$ is proper and grows quadratically in the geodesic distance.

On a shrinker the natural Laplace-type operator is the $f$-Laplacian
\[
\Delta_f u:=\Delta u-\langle\nabla f,\nabla u\rangle,
\]
which is self-adjoint and non-positive on the weighted space $L^2(e^{-f}dv)$. Hein--Naber~\cite{HeinNaber} and Cheng--Zhou~\cite{ChengZhou} proved that $\Delta_f$ has discrete spectrum and that the first non-zero eigenvalue satisfies $\lambda_1(-\Delta_f) \geq \tfrac{1}{2}$, with equality only when $(M,g)$ splits isometrically.  On K\"ahler Ricci shrinkers the spectral data of $\Delta_f$ are closely related to holomorphic functions of polynomial growth~\cite{HeOuDimension}.

Estimates for higher eigenvalues of the $f$-Laplacian on general Ricci shrinkers, however, remain absent from the literature.  The difficulty is that the
manifold is non-compact with possibly unbounded curvature, so classical techniques for compact manifolds or for manifolds with bounded geometry
do not apply directly.  The present paper fills this gap. We prove an upper bound for $\lambda_k(-\Delta_f)$ and a lower bound for the eigenvalues of a larger class of $f$-Laplace-type operators acting on sections of vector bundles, both are sharp in the exponent. As an application we obtain explicit bounds for Betti numbers using the Hodge isomorphism theorem in \cite{HeL2}.

Throughout the article we assume $n\ge 3$ and denote the sublevel sets of $f$ by
\[
\Omega_\lambda:=\{x\in M:f(x)\le\lambda\}.
\]
This is called the classically allowed region in the language of semiclassical analysis. Our first main result gives an upper bound for $\lambda_k(-\Delta_f)$ in terms of the volume growth of $\Omega_\lambda$.

\begin{theorem}[Upper bound for eigenvalues]\label{thm:intro-upper}
Let $(M^n,g,f)$ be a complete gradient Ricci shrinker.  Assume there exist $c_0>0$ and $\alpha\in[1,n]$ such that
\[
\Vol_g(\Omega_\lambda)\ge c_0\,\lambda^{\alpha/2},\qquad \lambda\ge 1+\frac n2.
\]
Then for all $k\ge 1$,
\[
\lambda_k(-\Delta_f)\le C(n,c_0)\,k^{\frac{2}{n+\alpha}} .
\]
\end{theorem}

By Cao--Zhou~\cite{CaoZhou} and Haslhofer--M\"uller~\cite{HaslhoferMuller} we always have $\Vol_g(\Omega_\lambda)\le C(n)\,\lambda^{n/2}$, and by
Munteanu--Wang~\cite{MunteanuWang2012} we have the lower bound $\Vol_g(\Omega_\lambda)\ge C(n)\,\lambda^{1/2}$.  Thus the parameter range
$\alpha\in[1,n]$ is natural; the hypothesis always holds with $\alpha=1$. (The condition $\lambda\ge 1+\frac n2$ guarantees
$\Omega_\lambda\neq\emptyset$ by~\eqref{eq:minf}.)

To obtain lower bounds for the eigenvalues---and to reach the topological applications described below---we work in greater generality.  Let $E\to M$ be a Hermitian vector bundle with a compatible connection $\nabla$, and let $A$ be a symmetric endomorphism of $E$.  Consider the operator
\[
L:=-\Delta+\nabla_{\nabla f}+A .
\]
$L$ is self-adjoint on $L^2(e^{-f}dv,\Gamma(E))$, and its spectrum is discrete~\cite{HeOu} provided $A$ is bounded from below.  We denote its
eigenvalues by $\lambda_1(L)\le\lambda_2(L)\le\cdots\to+\infty$ and its spectral counting function by $\mathcal{N}_L(\lambda):=\#\{k:\lambda_k(L)\le\lambda\}$.

\begin{theorem}[Counting function estimate]\label{thm:intro-counting}
Let $(M^n,g,f)$ be a complete gradient Ricci shrinker, let $E\to M$ be a Hermitian vector bundle with compatible connection, and let $L=-\Delta+\nabla_{\nabla f}+A$ where $\langle A(u),u\rangle\ge -C_A|u|^2$ for some $C_A\ge0$.  Assume there exists $\beta\in[1,n]$ such that
\[
\Vol_g(\Omega_\lambda)\le C(n)\,\lambda^{\beta/2},\qquad \lambda>0.
\]
Then
\[
\mathcal{N}_L(\lambda)\le C\,\rank(E)\,(1+\lambda+C_A)^{\frac{n+\beta}{2}},
\qquad\lambda\ge0,
\]
where $C$ depends only on $n$. 
\end{theorem}

By \cite{CaoZhou} and \cite{HaslhoferMuller} the volume growth assumption always holds with $\beta = n$. This improves the counting function estimate in~\cite{HeOu}. It is well-known that upper bounds of the counting function are equivalent to lower bounds of eigenvalues. When applied to the trivial line bundle $E=M\times\mathbb{R}$ with $A=0$ (so that $L=-\Delta_f$), Theorem~\ref{thm:intro-counting} immediately yields a lower bound for the eigenvalues of the $f$-Laplacian.

\begin{corollary}[Lower bound for eigenvalues]\label{cor:intro-lower}
Let $(M^n,g,f)$ be a complete gradient Ricci shrinker.  Assume there exists $\beta\in[1,n]$ such that
\[
\Vol_g(\Omega_\lambda)\le C(n)\,\lambda^{\beta/2}, \qquad \lambda>0.
\]
Then for all $k\ge 1$,
\[
\lambda_k(-\Delta_f)\ge c(n)\,k^{\frac{2}{n+\beta}}.
\]
\end{corollary}

Both our upper and lower bounds are sharp in the exponent. The $f$-Laplacian is unitarily equivalent to a Schr\"odinger-type operator (Section~\ref{sec:agmon}); consequently, on Ricci shrinkers with bounded geometry---or satisfying the criteria of
Braverman--Dai--Yan~\cite{BravermanDaiYan}---the Weyl law of Dai--Yan~\cite{DaiYan} applies.  On the product shrinker $\mathbb{S}^{n-p}(r)\times\mathbb{R}^p$ with $r=\sqrt{2(n-p-1)}$ and $1\le p\le n-2$, this gives
$\lambda_k(-\Delta_f)\sim c(n,p)\,k^{2/(n+p)}$.  On the Gaussian shrinker
$(\mathbb{R}^n,g_{\mathrm{Eucl}},f=|x|^2/4)$, $\Delta_f$ is the Ornstein--Uhlenbeck operator, for which $\lambda_k\sim c(n)\,k^{1/n}$.  In both cases the exponents match Theorem~\ref{thm:intro-upper} (with $\alpha=p$ and $\alpha=n$) and Corollary~\ref{cor:intro-lower} (with $\beta=p$ and $\beta=n$).

The key ingredient in the proofs of both theorems is the local conformal transform of Li--Li--Wang~\cite{LiLiWang}, which originates in an idea of Zhang~\cite{ZZhang}. Near any point $q\in\Omega_\lambda$, the conformal metric $\bar g=e^{-\frac{2}{n-2}(f-f(q))}g$ has uniformly bounded geometry on balls of radius $\sim\lambda^{-1/2}$; from this we obtain uniform volume doubling and a local Sobolev inequality (Section~\ref{sec:conformal}).  For the upper bound we combine the local Sobolev inequality with the classical method of Li--Yau~\cite{LiYau}.  For the counting function estimate, we assemble several classical ideas.  First, we prove Agmon-type estimates showing that eigensections are exponentially localized to the classically allowed region. Then, an IMS-type localization argument reduces the problem to estimates of a packing number, and of local Dirichlet counting functions.  Finally, we use the Cheng--Li heat-kernel method~\cite{ChengLi}, powered by the local Sobolev inequality, to estimate the counting function for the Dirichlet eigenvalues of $L$ on balls of radius $\sim\lambda^{-1/2}$ inside the classically allowed region.

The Agmon estimate developed for Theorem~\ref{thm:intro-counting} also yields a pointwise polynomial growth bound for eigensections on Ricci shrinkers without additional curvature assumption; we include this for its independent interest.
\begin{theorem}[Polynomial growth of eigensections]
\label{thm:polynomial-growth}
Let $(M^n,g,f)$ be a complete gradient Ricci shrinker, let $O\in M$ be a minimum
point of $f$, and set $r(x)=\mathrm{dist}(x,O)$. Let $L=-\Delta+\nabla_{\nabla f}+A$
act on sections of a Hermitian vector bundle, where $\langle A(u),u\rangle\ge -C_A|u|^2$ for some $C_A\ge0$. Suppose
$u\in W^{1,2}(e^{-f}dv)$ satisfies $\langle u,Lu\rangle\le\lambda|u|^2$.
Then for any $\sigma>0$,
\[
|u|(x)\le C\, \|u\|_{L^2(\Omega_R; d v)} \, (1+r(x))^{\,n+2 C_A+2\lambda+2\sigma},\qquad\forall x\in M,
\]
where $C$ depends only on $n$, $\sigma$, $\lambda$, $C_A$ and Perelman's entropy
$\mu$, and $R$ depends only on $n, C_A, \lambda$ and $\sigma$.
\end{theorem}
Colding--Minicozzi~\cite{ColdingMinicozzi} proved a sharp polynomial growth bound for the weighted level-set integral $r^{1-n}\int_{b=r}|u|^2|\nabla b|$ with $b=2\sqrt{f}$. Since $|\D b|^2 = 1 - S/f$, if the scalar curvature satisfies $\Sc < (1-\epsilon) f$, and if there is a uniform mean-value inequality, this can be turned into a pointwise estimate with sharp exponent. Theorem~\ref{thm:polynomial-growth} is complementary: it requires no curvature hypothesis, at the cost of sharpness in the exponent.

The main motivation for working with operators on vector bundles is that it covers some geometrically significant cases. The first instance is the $f$-Hodge Laplacian
\[
\Delta_f^{\dd}\omega = -\Delta\omega + \nabla_{\nabla f}\omega +\mathcal{R}_f(\omega),
\]
acting on $p$-forms, where $\mathcal{R}_f$ is the weighted curvature term \cite[Lemma~2.2]{HeL2}.  Under the assumption $|\Ric|\le\frac{1}{5n}f$ at infinity, the Hodge isomorphism theorem of~\cite[Theorem~1.4]{HeL2} identifies the space of $L^2(e^{-f}dv)$ $f$-harmonic $p$-forms with the de~Rham cohomology group $H^p_{\mathrm{dR}}(M)$, so that $\dim\Ker(\Delta_f^{\dd}) = b_p(M)$.  Applying Theorem~\ref{thm:intro-counting} to $L=\Delta_f^{\dd}$ with $\lambda=0$ yields the following.

\begin{corollary}[Betti number estimates]\label{cor:intro-betti}
Let $(M^n,g,f)$ be a complete gradient Ricci shrinker.  Suppose $|\Ric|\le\frac{1}{5n}f$ when $f$ is sufficiently large, and
$\langle\mathcal{R}_f(\omega),\omega\rangle\ge -K|\omega|^2$ for every $p$-form $\omega$.  Assume also there exists $\beta\in[1,n]$ such that
\[
\Vol_g(\Omega_\lambda)\le C(n)\,\lambda^{\beta/2},\qquad \lambda>0.
\]
Then
\[
b_p(M) \le C(n)\binom{n}{p}(1+K)^{\frac{n+\beta}{2}} .
\]
\end{corollary}

Compared with the Betti number estimate in~\cite[Theorem~1.1]{HeL2}, the curvature assumption here is stronger, but the estimate is more explicit in $K$ and does not involve the entropy. Note that $b_{n-1}(M)$ gives an upper bound for the number of ends, and the corresponding curvature term is the Einstein tensor (see \cite[section 5.3]{HeL2}). It remains an open problem whether a gradient Ricci shrinker must have only finitely many ends; see Munteanu-Wang~\cite{MunteanuWang2015, MunteanuWang2022}, Munteanu–Schulze–Wang~\cite{MSW}, Bertellotti–Buzano~\cite{BertellottiBuzano} for recent progress.

Another instance is the $f$-Lichnerowicz Laplacian in \cite{CaoZhu}
\[
\mathcal{L}_f h = \tfrac12\Delta h - \tfrac12\nabla_{\nabla f}h + \Rm(h,\cdot),
\]
acting on symmetric $2$-tensors $h$ belonging to $\Ker(\divg_f)_0$, which is the subspace of weighted divergence-free tensors orthogonal to $\Ric$ and to the image of $\divg_f^*$.  Cao and Zhu \cite{CaoZhu} proved that a compact gradient Ricci shrinker is linearly stable if and only if
$\mathcal{L}_f\le 0$ on this subspace.  In the non-compact case, if $\langle\Rm(h),h\rangle\le K|h|^2$ for every $h\in \Ker(\divg_f)_0$, then $\mathcal{L}_f$ also has discrete spectrum. Theorem~\ref{thm:intro-counting} with $L = -\mathcal{L}_f$ and $\lambda = 0$ provides an upper bound on the number of
unstable and null directions (nonnegative eigenvalues of $\mathcal{L}_f$) in terms of $n$ and $K$.

The paper is organized as follows.  Section~\ref{sec:conformal} recalls the local conformal transform and its consequences.
Section~\ref{sec:upperbound} proves Theorem~\ref{thm:intro-upper}. Section~\ref{sec:agmon} develops Agmon estimates and proves Theorem~\ref{thm:polynomial-growth}.  Section~\ref{sec:spectralcounting} proves Theorem~\ref{thm:intro-counting} and Corollaries~\ref{cor:intro-lower} and~\ref{cor:intro-betti}.

\medskip\noindent
\textbf{Acknowledgements:} The author would like to thank Prof. Detang Zhou for stimulating discussion.

% =====================================================
\section{Local conformal transform and its consequences}
\label{sec:conformal}

This section collects facts about gradient Ricci shrinkers that will be used throughout the paper. The local conformal transform and its geometric
consequences---metric equivalence, Ricci curvature bounds, volume doubling---are due to Li--Li--Wang~\cite{LiLiWang}, building on an earlier
observation of Zhang~\cite{ZZhang}.  The volume lower bound is a consequence of the log-Sobolev inequality of Li--Wang~\cite{LiWangheatkernel}.  We present these results in a self-contained manner for the reader's convenience.

Let $(M^n,g,f)$ be a complete gradient Ricci shrinker:
\begin{equation}\label{eq:shrinker}
\Ric+\nabla^2 f=\tfrac12 g .
\end{equation}
Taking the trace gives $\Delta f+S=\tfrac n2$.  The contracted second Bianchi identity together with~\eqref{eq:shrinker} yields $\nabla S=2\Ric(\nabla f,\cdot)$, which integrates to $S+|\nabla f|^2=f+c$.  Adding $c$ to $f$ gives the normalization
\begin{equation}\label{eq:normalize}
|\nabla f|^2+S=f .
\end{equation}
From these we obtain the identity
\begin{equation}\label{eq:identity}
\Delta f-|\nabla f|^2=\frac n2-f .
\end{equation}
Chen~\cite{Chen} proved $S\ge0$, hence $|\nabla f|\le\sqrt f$. Differentiating $\sqrt f$ along a geodesic $\gamma(s)$ gives
$|\tfrac d {ds} \sqrt f|=|\langle\nabla f,\gamma'\rangle|/(2\sqrt f)\le\frac12$, hence
\begin{equation}\label{eq:sqrtf}
\sqrt{f(\gamma(s))}\le\sqrt{f(\gamma(0))}+\frac s2 .
\end{equation}

Cao--Zhou~\cite{CaoZhou} established the quadratic growth estimate
\begin{equation}\label{eq:growthf}
\frac14(r(x)-c_1)_+^2\le f(x)\le\frac14(r(x)+c_2)^2,
\end{equation}
where $r(x)=\dist(x, p)$ for a minimal point $p$ of $f$, and the constants $c_1, c_2$ depend only on the dimension $n$. Equations \eqref{eq:normalize} and \eqref{eq:identity} imply that
\begin{equation}\label{eq:minf}
0\le \min_M f \le \frac{n}{2}.
\end{equation} 
Cao--Zhou~\cite{CaoZhou} also established the volume bound
\begin{equation}\label{eq:volgrowth}
\Vol(B(p,r))\le C r^n.
\end{equation}
By Haslhofer--M\"uller \cite{HaslhoferMuller}, the constant $C$ only depends on $n$ and \eqref{eq:volgrowth} holds for all $r> 0$.

\medskip\noindent
\textbf{The local conformal metric.}
Following Li--Li--Wang~\cite{LiLiWang}, for any $q\in M$ set
\begin{equation}\label{eq:conformal}
\bar g_q=e^{-\frac{2}{n-2}(f-f(q))}g .
\end{equation}
For $\lambda \ge \tfrac12$, we work on balls $B_g(q,r)$ with $r\le 1/\sqrt\lambda$ and $q$ in the sublevel set
\[
\Omega_\lambda:=\{x:f(x)\le\lambda\}.
\]

\begin{lemma}[Oscillation control]\label{lem:f-control}
For $q\in\Omega_\lambda$, $r=1/\sqrt\lambda$, $\lambda \ge \tfrac12$, and any
$y\in B_g(q,r)$,
\begin{align}
f(y)&\le\lambda+1+\frac1{4\lambda}\le\lambda+2,\\
|\nabla f(y)|&\le\sqrt{f(y)}\le\sqrt\lambda+\frac1{2\sqrt\lambda},\\
|f(y)-f(q)|&\le1+\frac1{2\lambda} \le 2 .
\end{align}
\end{lemma}

\begin{proof}
From~\eqref{eq:sqrtf}, $\sqrt{f(y)}\le\sqrt{f(q)}+r/2\le\sqrt\lambda+1/(2\sqrt\lambda)$. Squaring gives the first bound.  The second follows from~\eqref{eq:normalize} and $S\ge0$.  The third uses
$|f(y)-f(q)|\le r \sup_{B_g(q,r)}|\nabla f|\le(\sqrt\lambda+1/(2\sqrt\lambda))/\sqrt\lambda
=1+1/(2\lambda)\le 2$.
\end{proof}

\begin{corollary}[Metric equivalence]\label{cor:equiv}
Let $\lambda \ge \tfrac12$, $q\in\Omega_\lambda$. On $B_g(q,1/\sqrt\lambda)$, we have 
\[
C_0^{-2} g\le\bar g_q\le C_0^2 g,
\]
where $C_0 = e^{\frac{2}{n-2}}$. Consequently, for any $r\le1/\sqrt\lambda$, 
\[
B_{\bar g_q}(q, C_0^{-1} r)\subset B_g(q,r)\subset B_{\bar g_q}(q, C_0r).
\]
\end{corollary}

\begin{lemma}[Conformal Ricci bound]\label{lem:Ricci-bound}
Let $q\in\Omega_\lambda$, $r=1/\sqrt\lambda$, $\lambda \ge \tfrac12$. On $B_g(q,r)$, the Ricci curvature of $\bar g_q$ satisfies
\[
\overline{\Ric}\ge-C_1(n)\,\lambda\,\bar g_q .
\]
\end{lemma}

\begin{proof}
By direct calculation, the Ricci tensor of the conformal metric $\bar g_q$ is given by
\[
\overline{\Ric}
=\Ric+\nabla^2 f+\frac{df\otimes df}{n-2}
+\frac{\Delta f-|\nabla f|^2}{n-2}\,g .
\]
Using $\Ric+\nabla^2 f=\frac12g$ and~\eqref{eq:identity},
\[
\overline{\Ric}=\frac{df\otimes df}{n-2}+\frac{n-1-f}{n-2}\,g .
\]
Remarkably, the shrinker equation causes all original curvature terms to cancel, leaving an expression controlled purely by the potential function $f$.
For any tangent vector $X$, $df(X)^2\ge0$, so
\[
\overline{\Ric}(X,X)\ge\frac{n-1-f}{n-2}\,g(X,X)
=\frac{n-1-f}{n-2}\,e^{\frac{2(f-f(q))}{n-2}}\,\bar g_q(X,X).
\]
For $y\in B_g(q,r)$, $f(y)\le\lambda+2$ and $|f(y)-f(q)|\le2$ by Lemma~\ref{lem:f-control}.  Thus for a $\bar g_q$-unit $X$,
$\overline{\Ric}(X,X)\ge (n-2)^{-1} (n-1-\lambda-2)\,e^{\frac{4}{n-2}} \ge-C_1\lambda$ with $C_1=C_1(n)$.
\end{proof}

\begin{remark}
We actually get two-sided bounds for $\overline{\Ric}$, however, we only need the lower bound in this article. If we only have $\Ric + \D^2 f \ge \tfrac12 g$, additional assumption on $\Delta f - |\D f|^2$ is needed to bound $\overline{\Ric}$ from below.
\end{remark}

\medskip\noindent
\textbf{Local volume estimates.}
Throughout the article we will denote by $d\mu = e^{-f}dv$ the weighted measure. Let $\mu(B)=\int_B d\mu=\int_B e^{-f}dv$ be the weighted volume.

\begin{lemma}[Local doubling]\label{lem:doubling}
There exists $C_2=C_2(n)$ such that for all $q\in\Omega_\lambda$, $\lambda\ge\tfrac12$, $r\le e^{-\frac{2}{n-2}}/ \sqrt{\lambda}$,
\[
\frac{\Vol_g(B_g(q,r))}{\Vol_g(B_g(q,r/2))}\le C_2 \qquad \text{and} \qquad \frac{\mu(B_g(q,r))}{\mu(B_g(q,r/2))}\le C_2.
\]
\end{lemma}

\begin{proof}
By Lemma~\ref{lem:f-control}, $|f(y)-f(q)|\le2$ on $B_g(q,r)$, so $e^{-2}e^{-f(q)}dv\le d\mu\le e^{2}e^{-f(q)}dv$.  Hence
\[
\frac{\mu(B_g(q,r))}{\mu(B_g(q,r/2))}\le e^4\,
\frac{\Vol_g(B_g(q,r))}{\Vol_g(B_g(q,r/2))}.
\]
By Corollary~\ref{cor:equiv}, the $g$-volume ratio is comparable to the
$\bar g_q$-volume ratio. With $C_0 = e^{\frac{2}{n-2}}$, we have 
\[
\frac{\Vol_g(B_g(q,r))}{\Vol_g(B_g(q,r/2))} \le \frac{\Vol_{\bar g_q}(B_{\bar g_q}(q, C_0 r)) }{\Vol_{\bar g_q}(B_{\bar g_q}(q, C_0^{-1} r/2))}.
\]
Set $\tilde r= C_0 r$, apply the Bishop--Gromov volume comparison theorem on $\bar g_q$ with $\overline{\Ric}\ge-C_1\lambda$ (Lemma~\ref{lem:Ricci-bound}):
\[
\frac{\Vol_{\bar g_q}(B_{\bar g_q}(q,\tilde r))}
{\Vol_{\bar g_q}(B_{\bar g_q}(q,C_0^{-2}\tilde r/2))}
\le\frac{V_{-C_1\lambda}(\tilde r)}{V_{-C_1\lambda}(C_0^{-2} \tilde r/2)},
\]
where
$V_{-(n-1)K}(\rho)=\omega_{n-1}\int_0^\rho(\frac{\sinh\sqrt K t}{\sqrt K})^{n-1}dt$.
Since $r \le 1/\sqrt\lambda$, $\sqrt{C_1\lambda}\,\tilde r$ is uniformly bounded, so the model-space volume ratio depends only on $n$.
Chaining these inequalities together gives the result.
\end{proof}

The conformal Ricci lower bound yields a local volume upper bound.
\begin{lemma}[Volume upper bound]\label{lem:volupper}
  Let $\lambda \ge \tfrac12$, $x\in\Omega_{\lambda}$, and set
  $r \le e^{-2/(n-2)}/\sqrt{\lambda}$.
  Then
  \begin{equation}\label{eq:volupper}
    \Vol_g\!\bigl(B_g(x,r)\bigr)\le C(n)\,r^{\,n}.
  \end{equation}
\end{lemma}
\begin{proof}
By Lemma \ref{lem:Ricci-bound} and the volume comparison theorem, $\Vol_{\bar{g}}(B_{\bar g}(x, C_0 r)) \leq C(n) r^n$. The result then follows from Corollary \ref{cor:equiv}.
\end{proof}

The noncollapsing property of gradient Ricci shrinkers \cite{CarrilloNi} yields a local volume lower bound. Let
\begin{equation}\label{eq:entropy}
\bm{\mu} = \ln \int_M (4\pi)^{-\frac n 2 } e^{-f} dv =\ln \left( (4\pi)^{-\frac n 2 } \mu(M) \right).
\end{equation}
 $\bm{\mu}$ is called Perelman's entropy for the Ricci shrinker. 
\begin{lemma}[Volume lower bound]\label{lem:vollower}
  Let $\lambda \ge \tfrac12$, $x\in\Omega_{\lambda}$, and let
  $r \le e^{-2/(n-2)}/\sqrt{\lambda}$.
  Then
  \begin{equation}\label{eq:vollower}
    \Vol_g\!\bigl(B_g(x,r)\bigr)\ge v\,r^{\,n},
  \end{equation}
  with $v= c(n)e^{\bm \mu} >0$.
\end{lemma}
\begin{proof}
Following the argument of \cite[Lemma 2.4]{LiLiWang}, let $\eta$ be a cut-off function supported on $B_g(x, r)$, such that $\eta = 1$ on $B_g(x, r/2)$ and $|\D \eta|\leq \tfrac 4 r$. Let $I := \int_M \eta^2 dv$ and set $\varphi = I^{-1/2} \eta$, then $\int_M \varphi^2 dv = 1$.

By the sharp Log-Sobolev inequality of \cite[Theorem 1]{LiWangheatkernel} applied to $\varphi$ with $\tau = r^2$, we have 
\begin{equation}\label{eq:logI}
\bm{\mu} + n + \frac{n}{2} \log(4 \pi r^2) \le \log I + (e^{-1} + 64) \frac{\Vol_g(B_g(x, r))}{I} + r^2 \sup_{B_g(x, r)} \Sc.
\end{equation}
By Lemma \ref{lem:f-control}, we have $\Sc \le f \le \lambda + 2$ on $B_g(x, r)$, thus $r^2 \Sc \leq C(n)$. Note that 
\[
\Vol_g(B_g(x, r/2)) \le I \le \Vol_g(B_g(x, r)),
\]
by Lemma \ref{lem:doubling} the ratio $\Vol_g(B_g(x, r))/I \le C(n)$. The inequality \eqref{eq:logI} then yields a lower bound for $I$. The volume lower bound then follows.  
\end{proof}

\medskip\noindent
\textbf{Weighted Sobolev inequality.} As a consequence we have Sobolev inequalities on geodesic balls with radius $\sim 1/\sqrt{\lambda}$ centered in $\Omega_\lambda$.

\begin{proposition}[Local Sobolev]\label{prop:Sobolev}
Let $(M^n, g, f)$ be a gradient Ricci shrinker satisfying \eqref{eq:shrinker}. There exists $C_S=C_S(n)$ such that for all $q\in\Omega_\lambda$, $\lambda\ge \tfrac12$, $r \le \tfrac12 C_0^{-1}/ \sqrt{\lambda}$, $C_0 = e^{2/(n-2)}$,
and all $\phi\in C^1_0(B_g(q,r))$,
\begin{equation}\label{eq:localSobolev}
\left( \int_{B_g(q, r)} \phi^{\frac{2n}{n-2}} dv \right)^{\frac{n-2}{n}} \leq \frac{C_S(n)r^2}{\Vol_g(B_g(q, r))^{2/n}} \int_{B_g(q, r)} |\D \phi|^2 dv,
\end{equation}
and also the weighted version
\begin{equation}\label{eq:localweightedSobolev}
\left( \int_{B_g(q, r)} \phi^{\frac{2n}{n-2}} d\mu \right)^{\frac{n-2}{n}} \leq \frac{C_S(n)r^2}{\mu(B_g(q, r))^{2/n}} \int_{B_g(q, r)} |\D \phi|^2 d\mu,
\end{equation}
where $n \geq 3$ and $d\mu = e^{-f} dv$.
\end{proposition}

\begin{proof}
 On the conformal ball $B_{\bar g_q}(q,2C_0 r)$ we have $\overline{\Ric}\ge -C_1(n)\lambda$ and $2C_0 r\le\sqrt\lambda$. By the method of Li--Schoen~\cite{LiSchoen} and Saloff-Coste~\cite{SaloffCoste1992}, we have
\[
\left( \int_{B_{\bar g_q}(q, C_0 r)} \phi^{\frac{2n}{n-2}} d\bar v \right)^{\frac{n-2}{n}} \leq \frac{C(n)r^2 }{\Vol_{\bar g_q}(B_{\bar g_q}(q, C_0 r))^{2/n}} \int_{B_{\bar g_q}(q, C_0 r)} |\bar{\D} \phi|^2 d\bar{v},
\]
for all $\phi \in C_0^1(B_{\bar g_q}(q, C_0 r))$.  By Corollary~\ref{cor:equiv} $B_g(q,r)\subset B_{\bar g_q}(q, C_0 r)$, where the two metrics are equivalent up to the constant $C_0$. This equivalence and the volume comparison theorem convert the above Sobolev inequality to \eqref{eq:localSobolev}. Then by the equivalence of the measures $e^{-2} e^{-f(q)} dv \le d\mu \le e^2 e^{-f(q)} dv$, we can convert  \eqref{eq:localSobolev} to \eqref{eq:localweightedSobolev} after cancelling $e^{-f(q)}$ on both sides of the equation. 
\end{proof}

\medskip\noindent
\textbf{Local mean-value inequality.} Using the local Sobolev inequality \eqref{eq:localweightedSobolev} we can derive a mean-value inequality. This is not needed for the eigenvalue estimates, but will be used for the polynomial growth estimate in Corollary \ref{cor:polynomialgrowth}.

\begin{lemma}[Mean-value inequality]\label{lem:meanvalue-sharp}
Let $(M^n,g,f)$ be a complete gradient Ricci shrinker.  Assume $x\in\Omega_{k a}$ for some $k\ge 1$ and $a\ge \tfrac12$. Let $u\ge0$
be a nonnegative solution of $\Delta_f u\ge -a u$ on $B_g(x,r)$ with
$r\le \tfrac{1}{2}C_0^{-1}/\sqrt {ka}$.  
Then
\[
u(x)^2\le\frac{C_m(n)}{\Vol_g(B_g(x,r))}\int_{B_g(x,r)}u^2\,dv .
\]
\end{lemma}

\begin{proof}
We perform the Moser iteration entirely with respect to $d\mu=e^{-f}dv$.

Let $\varphi\in C^\infty_0(B_g(x,r))$ with $0\le\varphi\le1$, $\varphi\equiv1$ on $B_g(x,r/2)$, $|\nabla\varphi|\le4/r$.
For $p\ge1$, multiply $\Delta_f u\ge -a u$ by $u^{2p-1}\varphi^2$ and integrate against $d\mu$:
\begin{equation}\label{eq:step1}
a\int u^{2p}\varphi^2\,d\mu
\ge - \int\varphi^2 u^{2p-1}\Delta_f u \, d\mu .
\end{equation}
Since $\Delta_f$ is the natural Laplacian for the Dirichlet form on $L^2(d\mu)$,  integrating the RHS by parts yields
\begin{align*}
a\int u^{2p}\varphi^2\,d\mu
&\ge\int\bigl\langle\nabla(\varphi^2 u^{2p-1}),\nabla u\bigr\rangle\,d\mu\\
&=(2p-1)\int\varphi^2 u^{2p-2}|\nabla u|^2\,d\mu
+2\int\varphi u^{2p-1}\langle\nabla\varphi,\nabla u\rangle\,d\mu .
\end{align*}
By Cauchy--Schwarz,
\[
2\varphi u^{2p-1}|\langle\nabla\varphi,\nabla u\rangle|
\le\frac{2p-1}{2}\varphi^2 u^{2p-2}|\nabla u|^2 +\frac{2}{2p-1}|\nabla\varphi|^2 u^{2p}.
\]
Hence
\begin{equation}\label{eq:energy}
\frac{2p-1}{2}\int\varphi^2 u^{2p-2}|\nabla u|^2\,d\mu
\le a\int u^{2p}\varphi^2\,d\mu
+\frac{2}{2p-1}\int|\nabla\varphi|^2 u^{2p}\,d\mu .
\end{equation}
Now $u^{2p-2}|\nabla u|^2=p^{-2}|\nabla(u^{p})|^2$.  From~\eqref{eq:energy},
\[
\frac{2p-1}{2p^2}\int\varphi^2|\nabla(u^{p})|^2\,d\mu
\le a\int u^{2p}\varphi^2\,d\mu
+\frac{2}{2p-1}\int|\nabla\varphi|^2 u^{2p}\,d\mu .
\]
For $p\ge1$, $(2p-1)/(2p^2)\ge1/(2p)$ (since $2p-1\ge p$).
Thus
\[
\int\varphi^2|\nabla(u^{p})|^2\,d\mu
\le 2p a\int u^{2p}\varphi^2\,d\mu
+\frac{4p}{2p-1}\int|\nabla\varphi|^2 u^{2p}\,d\mu .
\]
With $|\nabla\varphi|^2\le \tfrac{16}{r^2}$ and $4p/(2p-1)\le4$, we obtain
\begin{equation}\label{eq:reversePoincare}
\begin{split}
\int|\nabla(u^{p}\varphi)|^2\,d\mu
&\le2\int\varphi^2|\nabla(u^{p})|^2\,d\mu+2\int u^{2p}|\nabla\varphi|^2\,d\mu\\
&\le\left( 4p a+\frac{160}{r^2}\right)\int_{B_g(x,r)}u^{2p}\,d\mu.
\end{split}
\end{equation}

Apply Proposition~\ref{prop:Sobolev} with
$\phi=u^{p}\varphi$ and $\lambda$ replaced by $ka$.
Setting $B=B_g(x,r)$, we have
\[
\left(\int (u^{p}\varphi)^{\frac{2n}{n-2}}d\mu\right)^{\frac{n-2}{n}}
\le\frac{C_S(n) r^2}{\mu(B)^{2/n}}
\int_{B} |\nabla(u^{p}\varphi)|^2\, d\mu .
\]
Using \eqref{eq:reversePoincare} and the choice of $\varphi$, 
\[
\left(\int_{B_{r/2}}u^{\frac{2np}{n-2}}d\mu\right)^{\frac{n-2}{n}}
\le\frac{C_S(n) }{\mu(B)^{2/n}} \bigl(4p ar^2+160\bigr) \int_{B}u^{2p}\,d\mu \le \frac{C(n) }{\mu(B)^{2/n}} \int_{B}u^{2p}\,d\mu,
\]
where we have used $r\sqrt a \leq \tfrac{1}{2\sqrt{k}C_0}$ and $k \geq 1$. 

Then standard Moser iteration yields
\begin{equation}\label{eq:meanvalue-mu}
u(x)^2\le\frac{C(n)}{\mu(B)}\int_{B}u^2\,d\mu,
\end{equation}
for a possibly different constant $C(n)$. 

On $B = B_g(x,r)$, Lemma~\ref{lem:f-control} gives $|f(y)-f(x)|\le2$, so $e^{-2}e^{-f(x)}\le e^{-f(y)}\le e^{2}e^{-f(x)}$ for all $y\in B_g(x,r)$.  Hence
\[
e^{-2}e^{-f(x)}\Vol_g(B)\le\mu(B)\le e^{2}e^{-f(x)}\Vol_g(B),
\]
and similarly
\[
e^{-2}e^{-f(x)}\!\int_{B}u^2\,dv \le\int_{B}u^2\,d\mu \le e^{2}e^{-f(x)}\!\int_{B}u^2\,dv .
\]
Substituting into~\eqref{eq:meanvalue-mu}, the factors $e^{-f(x)}$ cancel:
\[
u(x)^2\le\frac{C(n)\,e^4}{\Vol_g(B)}\int_{B}u^2\,dv .
\]
Setting $C_m=C(n) e^4$ completes the proof.
\end{proof}

\begin{remark}
In the original Lemma~2.2 of~\cite{HeOu}, the integration by parts was performed against $dv$ instead of $d\mu$, leaving the drift term
$\langle\nabla f,\nabla u\rangle$ to be estimated by $|\nabla f|\simeq \tfrac12 d(p,x)$.  This produced the superfluous factor
$(a+1+d(p,x))$ and the suboptimal exponent $n+1$.
\end{remark}

% =====================================================
\section{Upper bounds of higher eigenvalues}
\label{sec:upperbound}
Using the consequences of the local conformal transform in Section \ref{sec:conformal}, we prove the upper bound estimate for eigenvalues of $\Delta_f$ by the method of Li-Yau \cite{LiYau}. The following is a restatement of Theorem~\ref{thm:intro-upper}. 
\begin{theorem}\label{thm:upperbound}
Let $(M^n, g, f)$ be a complete gradient Ricci shrinker with $\Ric + \D^2 f= \frac12 g$.  Assume:
\begin{enumerate}
\item[(H1)] There exist $c_0>0$ and $\alpha \in [1,n]$, such that for all
$\lambda \ge 1+ \frac{n}{2}$,
\[
\Vol_g(\Omega_\lambda) \ge c_0\,\lambda^{\,\alpha/2},
\qquad \Omega_\lambda := \{x\in M : f(x) \le \lambda\}.
\]
\end{enumerate}
Then there exists a constant $C = C(n, c_0)$ such that for all $k\ge 1$,
\[
\lambda_k(-\Delta_f) \le C\, k^{2/(n+\alpha)}.
\]
\end{theorem}
By \eqref{eq:minf}, $\lambda \ge 1 + \tfrac n2$ ensures that $\Omega_\lambda$ is not empty. 
\begin{lemma}\label{lem:packing}
There exists $\eta = \eta(n) > 0$ such that for all sufficiently large $\lambda$, the set $\Omega_\lambda$ contains at least
\[
N = \bigl\lfloor \eta\, \lambda^{n/2}\, \Vol_g(\Omega_\lambda)
\bigr\rfloor
\]
points with pairwise $g$-distance at least $3/\sqrt\lambda$.
\end{lemma}
\begin{proof}
Let $r = 1/\sqrt\lambda$.  Take a maximal set $\{x_0,\dots,x_N\} \subset \Omega_\lambda$ with $\operatorname{dist}(x_i,x_j) \ge 3r$ for $i\neq j$.
By maximality, the balls $B_g(x_i, 3r)$ cover $\Omega_\lambda$.  Hence
\[
\Vol_g(\Omega_\lambda) \le \sum_{i=0}^N \Vol_g(B_g(x_i, 3r))
\le (N+1) \max_{0\le i\le N} \Vol_g(B_g(x_i, 3r)).
\]
By Lemma \ref{lem:volupper},
$\Vol_g(B_g(x_i, 3r)) \le C(n) \lambda^{-n/2}$ uniformly for
$x_i \in \Omega_\lambda$.  Thus
\[
N+1 \ge \frac{\Vol_g(\Omega_\lambda)}{C(n)\,\lambda^{-n/2}}
= C(n)^{-1}\,\lambda^{n/2}\,\Vol_g(\Omega_\lambda).
\]
Setting $\eta = 1/C(n)$ yields the claim.
\end{proof}

\begin{lemma}[Dirichlet eigenvalue]\label{lem:DirichletLambda1}
For all $q\in\Omega_\lambda$, $\lambda\ge \tfrac12$, $r= C_0^{-1}/ \sqrt\lambda$, $C_0 = e^{\frac{2}{n-2}}$,
\[
\lambda_1^{D,f}\!\bigl(B_g(q,r)\bigr)
:=\inf_{u\in C^\infty_0\setminus\{0\}}
\frac{\int_{B_g(q,r)}|\nabla u|^2\,d\mu}{\int_{B_g(q,r)}u^2\,d\mu}\le C_1(n)\,\lambda .
\]
\end{lemma}

\begin{proof}
Use the test function $\varphi(y)=\max\{0,1-\dist(q,y)/r\}$.  Then $|\nabla\varphi|\le1/r$ a.e., $\varphi\ge\frac12$
on $B_g(q,r/2)$, so
\[
\lambda_1^{D,f}\!\bigl(B_g(q,r)\bigr)\le(4/r^2)\, \frac{\mu(B_g(q,r))}{\mu(B_g(q,r/2))}
=4C_0^2\lambda\,\frac{\mu(B_g(q,r))}{\mu(B_g(q,r/2))}.
\]
Now apply Lemma~\ref{lem:doubling} and the estimate follows. 
\end{proof}

\begin{proof}[Proof of Theorem~\ref{thm:upperbound}]
Let $\eta$ be the constant from Lemma~\ref{lem:packing} and $c_0$ from $(H1)$, take 
\[
\eta' = \min\left\{ \eta, \frac{2}{c_0} \left(1+\frac{n}{2}\right)^{-n}\right\}.
\]
For an integer $k \ge 1$, set
\[
\lambda = \left(\frac{k+1}{\eta' c_0}\right)^{\!2/(n+\alpha)},
\]
so that $\eta' c_0\lambda^{(n+\alpha)/2} = k+1$. By the choice of $\eta'$ we have $\lambda \ge 1 +\tfrac n2 $, hence $(H1)$ applies. 

By Lemma~\ref{lem:packing} and hypothesis $(H1)$,
\[
N = \bigl\lfloor \eta\,\lambda^{n/2}\,
\Vol_g(\Omega_\lambda) \bigr\rfloor
\ge \bigl\lfloor \eta'\,\lambda^{n/2}\, c_0\lambda^{\alpha/2} \bigr\rfloor
= \bigl\lfloor \eta' c_0\lambda^{(n+\alpha)/2} \bigr\rfloor = k+1.
\]
Thus there exist points $x_0,\dots,x_k \in \Omega_\lambda$ with pairwise
distance $\ge 3r$, where $r = 1/\sqrt\lambda$.

The geodesic balls
\[
B_i := B_g(x_i, r), \qquad i = 0,\dots,k,
\]
are pairwise disjoint.  By Lemma~\ref{lem:DirichletLambda1}, $\lambda_1^{D,f}(B_i) \le C_1\lambda$ for each $i$.
By the min-max principle,
\[
\lambda_k(-\Delta_f) \le \max_{0\le i\le k} \lambda_1^{D,f}(B_i)
\le C_1\lambda
= C_1\,(\eta' c_0)^{-2/(n+\alpha)}\,(k+1)^{2/(n+\alpha)}
\le C\, k^{2/(n+\alpha)}
\]
with $C =\sup_{\alpha \in [1,n]} C_1\,(\eta' c_0)^{-2/(n+\alpha)}\, 2^{2/(n+\alpha)}$, which depends only on
$n$ and $c_0$.
\end{proof}

% =====================================================
\section{Agmon estimates and polynomial growth for eigensections}
\label{sec:agmon}

In this section we adapt the Agmon-type exponential decay estimates developed in~\cite{HeL2} for weighted $L^2$ harmonic forms to the setting of eigensections of the operator $L$.  The core idea, originating from Agmon~\cite{Agmon}, is to conjugate the operator by a $1$-parameter family of weights and exploit the positivity of the potential function.  

\medskip\noindent
\textbf{Setting.}
Let $E\to M$ be a Hermitian vector bundle of rank $\rank(E)$ with a compatible connection $\D$, and let $A\in\Gamma(\End(E))$ be a symmetric endomorphism satisfying
\begin{equation}\label{eq:Abound}
\langle A(\xi),\xi\rangle\ge -C_A|\xi|^2
\end{equation}
for some constant $C_A\ge0$.  The operator under study is
\begin{equation}\label{eq:L}
L:=-\Delta+\D_{\D f}+A,
\end{equation}
which is self-adjoint on the weighted space $L^2(d\mu,\Gamma(E))$ with $d\mu=e^{-f}dv$, and has discrete spectrum $\lambda_1\le\lambda_2\le\cdots$
bounded below~\cite{HeOu}.  Set $\mathcal{N}_L(\lambda):=\#\{i:\lambda_i\le\lambda\}$.

\medskip\noindent
\textbf{Conjugation to Schr\"odinger operator.}
The operator $L$ in~\eqref{eq:L} is self-adjoint on the weighted space $L^2(d\mu,\Gamma(E))$ with $d\mu=e^{-f}dv$.  To apply the Agmon method, we first conjugate by $e^{-f/2}$ to obtain a Schr\"odinger-type operator on the unweighted space $L^2(dv,\Gamma(E))$.  Direct computation shows that the first-order drift term cancels exactly, leaving only a zeroth-order potential.

\begin{lemma}[Conjugation]\label{lem:conjugation}
Define the unitary map $U:L^2(d\mu,\Gamma(E))\to L^2(dv,\Gamma(E))$ by $Uu=e^{-f/2}u$.  Then the conjugated operator
\[
\widetilde L:=U L U^{-1}=e^{-f/2}L\,e^{f/2}
\]
takes the Schr\"odinger form
\begin{equation}\label{eq:tildeL}
\widetilde L=-\Delta+V_f+A,
\qquad\text{where}\qquad
V_f:=\frac14|\nabla f|^2-\frac12\Delta f .
\end{equation}
By the identities \eqref{eq:normalize} and \eqref{eq:identity} on gradient Ricci shrinkers, we can write $V_f = \tfrac{1}{4} (f + \Sc -n)$.
\end{lemma}

\begin{proof}
For any smooth section $v$ and function $\phi$, the rough Laplacian $\Delta=\tr\nabla^2$ satisfies the identity
\[
\Delta(\phi v)=(\Delta\phi)v+\phi\Delta v+2\nabla_{\nabla\phi}v .
\] 
Taking $\phi=e^{f/2}$, we have $\nabla\phi=\tfrac12 e^{f/2}\nabla f$ and
\[
\Delta\phi=\divg\!\bigl(\tfrac12 e^{f/2}\nabla f\bigr)
=\tfrac14 e^{f/2}|\nabla f|^2+\tfrac12 e^{f/2}\Delta f .
\]
Hence
\[
\Delta(e^{f/2}v)=e^{f/2}\left(\Delta v+\nabla_{\nabla f}v+\tfrac14|\nabla f|^2 v+\tfrac12\Delta f\,v\right).
\]
On the other hand,
\[
\nabla_{\nabla f}(e^{f/2}v)=\tfrac12 e^{f/2}|\nabla f|^2 v+e^{f/2}\nabla_{\nabla f}v .
\]
Putting these together, the $\nabla_{\nabla f}v$ terms cancel, and we obtain
\[
L(e^{f/2}v)=e^{f/2}\left(-\Delta v+\bigl(\tfrac14|\nabla f|^2-\tfrac12\Delta f\bigr)v+A(v)\right).
\]
 Multiplying by $e^{-f/2}$ yields~\eqref{eq:tildeL}.  
\end{proof}

Consequently, if $Lu=\lambda u$, then the transformed section $\tilde u:=e^{-f/2}u$ satisfies
\begin{equation}\label{eq:eigentransformed}
\tilde L \tilde u= \lambda \tilde u .
\end{equation}

Since $V_f = \tfrac{1}{4}(f+\Sc - n)$, by \eqref{eq:growthf} and $\Sc \ge 0$, $V_f$ is proper and grows to infinity. When $A$ is bounded below, the discreteness of the spectrum of $L$ follows from well-known spectral theory of Schr\"odinger operators \cite[page 120]{ReedSimon}.

\medskip\noindent
\textbf{Agmon estimate on shrinkers.}
We now prove that every eigensection of $\tilde L$ decays exponentially relative to $f$. On a shrinker, the required structural constants are determined entirely by the soliton equations, so the estimate takes a clean, explicit form.

\begin{lemma}\label{lem:agmon-shrinker}
Let $(M^n,g,f)$ be a complete gradient Ricci shrinker, let $d \mu = e^{-f} dv$ be the weighted measure, let $u \in W^{1,2}(d\mu)$ be a section such that $\langle u, L u \rangle \le \lambda |u|^2$, $\lambda \ge 0$. Let $\tilde u = e^{-f/2} u$. Suppose $\langle A(\xi),\xi\rangle \geq - C_A |\xi|^2$ for every section $\xi$.
Then for any $0 < \epsilon \le \tfrac 12$, setting
\[
R_\lambda = 2\left(1 + \frac{n}{4} + C_A + \lambda \right)/\epsilon, 
\]
we have 
\begin{equation}\label{eq:agmon-shrinker}
\int_{\{f>R_\lambda\}} e^{(1- 2 \epsilon) f}\,|\tilde u|^2\,dv
\le R_\lambda e^{(1-2\epsilon)R_\lambda} \int_{\{f\le R_\lambda\}}  |\tilde u|^2\,dv .
\end{equation}
\end{lemma}

\begin{proof}
From~\eqref{eq:tildeL} and the bound on $A$, we have the pointwise inequality
\begin{equation}\label{eq:effective-potential}
V_f|\xi|^2+\langle A(\xi),\xi\rangle-\lambda|\xi|^2
\ge \left(\frac{f}{4}-\frac{n}{4}-C_A-\lambda \right)|\xi|^2,
\end{equation}
 for every section $\xi$.

We first verify that $\tilde u \in W^{1,2}(dv)$, which is not obvious since $\D \tilde u = e^{-f/2} \D u - \tfrac12 \D f \otimes \tilde u$ and $|\D f|$ is unbounded. The assumption $\langle u, L u\rangle \leq \lambda |u|^2 $ implies that $\langle \tilde u, \tilde L \tilde u\rangle \leq \lambda |\tilde u|^2$. Multiply this inequality by the square of any cutoff function $\eta \in C_0^1(M)$, integration by parts yields
\[
\int \eta^2 |\D \tilde u|^2 + 2\eta \langle \D_{\D \eta} \tilde u, \tilde u\rangle + \eta^2 \langle (V_f + A - \lambda) \tilde u, \tilde u \rangle \, dv\le 0. 
\]
Calculate the second term 
\[
2\eta \langle \D_{\D \eta} \tilde u, \tilde u \rangle = 2\eta e^{-f} \langle \D_{\D \eta} u, u\rangle - \eta \langle \D \eta, \D f\rangle |\tilde u|^2 \ge -2 \eta e^{-f} |\D \eta| |u| |\D u| - \eta |\D \eta| |\D f| |\tilde u|^2.
\]
By \eqref{eq:effective-potential} and the fact that $|\D f|^2 = f - \Sc \le f$, we have 
\[
\begin{split}
\int \eta^2 |\D \tilde u|^2  +  \eta^2 & \left(\frac{f}{4}-\frac{n}{4} -C_A  - \lambda \right) |\tilde u|^2 \, dv \le  \int 2 \eta e^{-f} |\D \eta| |u| |\D u| + \eta |\D \eta| |\D f| |\tilde u|^2 \, dv\\
\le & \int \eta^2 e^{-f} |u|^2 + |\D \eta|^2 e^{-f} |\D u|^2 + |\D \eta|^2 |\tilde u|^2 + \frac{f}{4} \eta^2  |\tilde u|^2 \, dv. 
\end{split}
\]
The terms containing $f$ cancel on both sides of the inequality. Then by the assumption $u \in W^{1,2}(d\mu)$, we can let $\eta \to 1$ to see that $\int |\D \tilde u|^2 dv \le C \int |u|^2 d\mu$, hence $\tilde u \in W^{1,2}(dv)$.  

By \eqref{eq:growthf}, $f(x)\to\infty$ as $x\to\infty$.  Define the bounded regularization
\[
f_\theta:=\frac{f}{1+\theta f},\qquad \theta\in(0,1).
\]
Then $f_\theta\le1/\theta$ everywhere, $f_\theta(x)\to f(x)$ pointwise as $\theta\to0$, and
\[
|\nabla f_\theta|=\frac{|\nabla f|}{(1+\theta f)^2}\le|\nabla f| .
\]
From $|\nabla f|^2=f-S\le f$ we obtain $|\nabla f_\theta|^2\le f$.

Define the weight function
\[
\varphi(s):=e^{(\frac{1}{2} -\epsilon) s},\qquad s\ge0,
\]
which satisfies $\varphi'(s)=\bigl(\frac{1}{2}-\epsilon \bigr)\varphi(s)$.  Set $\tilde u=e^{-f/2}u$ and let
\[
\tilde u_\theta:=\varphi(f_\theta)\tilde u .
\]
Note that $f_\theta$ and $\varphi(f_\theta)$ are bounded on $M$.

The inequality $\langle u, L u \rangle \le  \lambda |u|^2$ gives $\langle (-\Delta+V_f+A-\lambda)(\varphi(f_\theta)^{-1}\tilde u_\theta), \varphi(f_\theta) \tilde u_\theta \rangle \le 0$.  Integrate by parts:
\[
0 \ge \int_M\langle\nabla(\varphi\tilde u_\theta),\nabla(\varphi^{-1}\tilde u_\theta)\rangle\,dv+ \int_M
(V_f-\lambda)|\tilde u_\theta|^2 + \langle A(\tilde u_\theta), \tilde u_ \theta \rangle \,dv .
\]
The integration by parts can be verified by standard cutoff arguments since $\varphi$ is bounded and $\tilde u\in W^{1,2}(dv)$. 

Now compute the gradient inner product.  Writing $\varphi=\varphi(f_\theta)$ for brevity,
\[
\nabla(\varphi\tilde u_\theta)=\varphi\nabla\tilde u_\theta+\varphi' \tilde u_\theta \D f_\theta,
\qquad
\nabla(\varphi^{-1}\tilde u_\theta)=\varphi^{-1}\nabla\tilde u_\theta-\varphi^{-2}\varphi'\tilde u_\theta \D f_\theta.
\]
Taking the inner product,
\begin{align*}
\langle\nabla(\varphi\tilde u_\theta),\nabla(\varphi^{-1}\tilde u_\theta)\rangle
&=|\nabla\tilde u_\theta|^2
-\varphi^{-1}\varphi'  \langle \tilde u_\theta, \nabla_{\D f_\theta} \tilde u_\theta \rangle\\
&\quad+\varphi^{-1}\varphi'  \langle \tilde u_\theta,\nabla_{\D f_\theta}\tilde u_\theta\rangle
-(\varphi^{-1}\varphi')^2|\D f_\theta|^2|\tilde u_\theta|^2 .
\end{align*}
The two middle terms cancel.  Since $\varphi^{-1} \varphi'=\frac12-\epsilon $, we obtain
\[
\langle\nabla(\varphi\tilde u_\theta),\nabla(\varphi^{-1}\tilde u_\theta)\rangle
=|\nabla\tilde u_\theta|^2
-\left(\frac12-\epsilon \right)^2|\nabla f_\theta|^2\,|\tilde u_\theta|^2 .
\]

Therefore the full identity reads
\begin{equation}\label{eq:energy-identity}
0\ge \int_M\left(V_f+\frac{\langle A(\tilde u_\theta),\tilde u_\theta\rangle}{|\tilde u_\theta|^2}
-\lambda-\Bigl(\frac12-\epsilon \Bigr)^2|\nabla f_\theta|^2\right)|\tilde u_\theta|^2\,dv
+\int_M|\nabla\tilde u_\theta|^2\,dv .
\end{equation}

Insert the lower bounds~\eqref{eq:effective-potential} and $|\nabla f_\theta|^2\le f$, and drop the nonnegative term $|\nabla\tilde u_\theta|^2$:
\begin{equation}\label{eq:L2estimatetheta}
\begin{split}
0\ge & \int_M\left(\frac14 f- \frac n4 -C_A-\lambda-\Bigl(\frac12-\epsilon \Bigr)^2 f\right)|\tilde u_\theta|^2\,dv \\
= & \int_M\left( (\epsilon - \epsilon^2) f - \frac n4 -C_A-\lambda \right)|\tilde u_\theta|^2\,dv.
\end{split}
\end{equation}

Denote $b = \tfrac n 4 + C_A + \lambda$. Let $\epsilon < \tfrac{1}{2}$, thus $\epsilon -\epsilon^2 > \tfrac12 \epsilon$. Let $R_\lambda : = 2(b+1)/\epsilon$. Since $\tilde u$ is bounded on $\{f < R_\lambda\}$,  
\[
\begin{split}
\lim_{\theta \to 0} \int_{\{ f < R_\lambda\}} \left((\epsilon - \epsilon^2) f - b \right)|\tilde u_\theta|^2\,dv = &  \int_{\{f < R_\lambda\}} \left(  (\epsilon - \epsilon^2) f- b\right) \varphi(f)^2|\tilde u|^2 \,dv. \\
\end{split}
\]
Then by \eqref{eq:L2estimatetheta} and Fatou's lemma, we have
\[
\int_{\{f \ge R_\lambda\}} \left( (\epsilon - \epsilon^2) f- b\right)  \varphi(f)^2 |\tilde u|^2 \,dv \leq  \int_{\{f < R_\lambda\}} \left( b - (\epsilon - \epsilon^2) f \right)_+   \varphi(f)^2 |\tilde u|^2 \, dv.
\]
This yields the lemma. 
\end{proof}

\medskip\noindent
\textbf{Pointwise polynomial growth estimate.}
Lemma~\ref{lem:agmon-shrinker} yields exponential decay for $\tilde u$, which only gives $|u| \lesssim e^{\epsilon f}$. To obtain polynomial control we use instead the weight $\varphi(s) = (1+s)^{-q}e^{s/2}$ from \cite[Lemma~3.1]{HeL2}.  The factor $e^{s/2}$ is the maximal exponent
allowed by the shrinker identities; it cancels the conjugation exactly, giving $\varphi(f)\tilde u = (1+f)^{-q}u$.  Unlike the harmonic case
in \cite{HeL2}, here $\lambda$ may be large, so $q = \sigma + \frac{n}{4} + C_A + \lambda$ and $R_\lambda$ grow with
$\lambda$; the following lemma verifies that the estimate remains valid with explicit $\lambda$-dependence. This is not needed for the eigenvalue estimates, but may be of independent interest.

\begin{lemma}\label{lem:agmon-shrinker-polynomial}
Let $(M^n,g,f)$ be a complete gradient Ricci shrinker, let $d \mu = e^{-f} dv$ be the weighted measure, let $u \in W^{1,2}(d\mu)$ be a section such that $\langle u, L u \rangle \le \lambda |u|^2$, $\lambda \ge 0$.  Suppose $\langle A(\xi),\xi\rangle \geq - C_A |\xi|^2$ for every section $\xi$.
Then for any $0 < \sigma \le \tfrac n 8$, setting
\[
q= \sigma + \frac{n}{4} + C_A + \lambda, \quad R_\lambda = 2\left(\frac{n}{4} + C_A + \lambda \right)^3/\sigma^2, 
\]
we have 
\begin{equation}\label{eq:agmon-shrinker-polynomial}
\frac{\sigma}{2}\int_{\{f>R_\lambda\}}(1+f)^{-2q}\,| u|^2\,dv
\le \left(\sigma + \frac{n}{4} + C_A + \lambda \right)^2 \int_{\{f\le R_\lambda\}}  | u|^2\,dv .
\end{equation}
\end{lemma}

\begin{proof}
Define $f_\theta$ as in the proof of Lemma \ref{lem:agmon-shrinker}.

Choose $q\ge0$ (to be fixed later).  Define the weight function
\[
\varphi(s):=(1+s)^{-q}e^{s/2},\qquad s\ge0,
\]
which satisfies $\varphi'(s)=\bigl(\frac{1}{2}-\frac{q}{1+s}\bigr)\varphi(s)$.  Set $\tilde u=e^{-f/2}u$ and let
\[
\tilde u_\theta:=\varphi(f_\theta)\tilde u .
\]
We still have $f_\theta$ and $\varphi(f_\theta)$ bounded on $M$.

By the same argument as before, and use $\varphi^{-1}\varphi'=\frac12-q/(1+f_\theta)$, we obtain
\[
\langle\nabla(\varphi\tilde u_\theta),\nabla(\varphi^{-1}\tilde u_\theta)\rangle
=|\nabla\tilde u_\theta|^2
-\left(\frac12-\frac{q}{1+f_\theta}\right)^2|\nabla f_\theta|^2\,|\tilde u_\theta|^2 .
\]
Then \eqref{eq:energy-identity} becomes
\[
0\ge \int_M\Bigl(V_f+\frac{\langle A(\tilde u_\theta),\tilde u_\theta\rangle}{|\tilde u_\theta|^2}
-\lambda-\left(\frac12-\frac{q}{1+f_\theta}\right)^2|\nabla f_\theta|^2\Bigr)|\tilde u_\theta|^2\,dv
+\int_M|\nabla\tilde u_\theta|^2\,dv .
\]
Insert the lower bounds~\eqref{eq:effective-potential} and $|\nabla f_\theta|^2\le f$, and drop the nonnegative term $|\nabla\tilde u_\theta|^2$:
\begin{equation}\label{eq:L2estimatetheta-polynomial}
\begin{split}
0\ge & \int_M\Bigl(\frac14 f- \frac n4 -C_A-\lambda-\left(\frac12-\frac{q}{1+f_\theta}\right)^2 f\Bigr)|\tilde u_\theta|^2\,dv \\
= & \int_M\left( \frac{qf(1+f_\theta - q)}{(1+f_\theta)^2} - \frac n4 -C_A-\lambda \right)|\tilde u_\theta|^2\,dv
\end{split}
\end{equation}
Now examine the limit as $f\to\infty$ and $\theta\to0$.  Since $f_\theta=f/(1+\theta f)$, we have $f/(1+f_\theta)\to1$.  The expression in the large parentheses tends to $q - \tfrac n 4 -C_A - \lambda$. 

Denote $b = \tfrac n 4 + C_A + \lambda$. Choose $q = \sigma + b$ as in the statement. Let $R_\lambda : = k b/\sigma$ for some $k > 1$ to be determined later. Since $\tilde u$ is bounded on $\{f < R_\lambda\}$,  
\[
\begin{split}
\lim_{\theta \to 0} \int_{\{ f < R_\lambda\}} \left( \frac{qf(1+f_\theta - q)}{(1+f_\theta)^2} - b \right)|\tilde u_\theta|^2\,dv = &  \int_{\{f < R_\lambda\}} \left(  \frac{qf(1+f-q)}{(1+f)^2} - b\right) \varphi(f)^2|\tilde u|^2 \,dv. \\
\end{split}
\]
Then by \eqref{eq:L2estimatetheta-polynomial} and Fatou's lemma, we have
\[
\int_{\{f \ge R_\lambda\}} \left( \frac{qf(1+f-q)}{(1+f)^2} - b\right)  \varphi(f)^2 |\tilde u|^2 \,dv \leq  \int_{\{f < R_\lambda\}} \left( b - \frac{qf(1+f-q)}{(1+f)^2} \right)_+   \varphi(f)^2 |\tilde u|^2 \, dv.
\]
Note that  $\varphi(f)^2 |\tilde u|^2 = (1+f)^{-2q} |u|^2$. On $\{f \geq R_\lambda = k b/\sigma \}$, we want 
\[
\frac{qf(1+f-q)}{(1+f)^2} - b \ge \frac{\sigma}{2}.
\]
Since $0 < \sigma \le \tfrac{n}{8}$ and $b \ge \tfrac n 4$, we always have $\tfrac{\sigma}{b} \le \tfrac 12$, hence $R_\lambda \ge 2k$. It suffices to have 
\[
q\, \frac{2k}{1+2k} \left( 1 - \frac{q}{1+2k}\right) >  \frac{q}{1+1/(2k)} \left( 1 - \frac{q}{2k}\right) \ge b + \frac{\sigma }{2}.
\]
Solving the last inequality yields $k \ge \sigma^{-1} (q^2 + b + \sigma /2)$, therefore it suffices to take $k = 2b^2 /\sigma$, and $R_\lambda = 2 b^3/\sigma^2$.

On $\{f < R_\lambda\}$, we have $\left( b - \frac{qf(1+f-q)}{(1+f)^2} \right)_+ \leq (b + \sigma)^2$. Hence we obtain the estimate in the lemma. 
\end{proof}

\begin{corollary}[Polynomial growth]\label{cor:polynomialgrowth}
Let $(M^n, g, f)$ be a complete gradient Ricci shrinker, let $d\mu = e^{-f} dv$ be the weighted measure. Let $u \in W^{1,2}(d\mu)$ be a section satisfying $\langle u, Lu \rangle \le \lambda |u|^2$ for the operator $L$ in \eqref{eq:L}. Then for any $\sigma > 0$, we have 
\[
|u|(x) \le C(n) \, \sigma^{-1} \, e^{-\bm \mu} \,  (1+\lambda + C_A)^2 \,  \|u\|_{L^2(\{f < R_\lambda \};\, dv)} \,  (1+f(x))^{n/2 + C_A + \lambda + \sigma},
\]
where $C_A$ and $R_\lambda$ are the same constants as in Lemma \ref{lem:agmon-shrinker-polynomial},  $\bm \mu$ is Perelman's entropy \eqref{eq:entropy}. 
\end{corollary}

\begin{proof}
Let $R_\lambda$ be the same constant as in Lemma \ref{lem:agmon-shrinker-polynomial}. For any $x\in M \setminus \Omega_{R_\lambda}$, set $R = f(x)$, $\Omega_R = \{f < R\}$ and $r = \tfrac12 e^{-2/(n-2)} /\sqrt{R}$. By Lemma \ref{lem:f-control}, $B_g(x, r) \subset \Omega_{R+2}$. Then Lemma \ref{lem:agmon-shrinker-polynomial} implies
\[
(R+2)^{-2q}\int_{B_g(x, r)} |u|^2 dv \le 2 \sigma^{-1} (n + \lambda + C_A)^2 \|u\|^2_{L^2(\Omega_{R_\lambda}; dv)}.
\]
By Kato's inequality we have $\Delta_f |u| \ge -(\lambda + C_A) |u|$. Then Lemma \ref{lem:meanvalue-sharp} implies that
\[
|u|^2(x) \le \frac{C(n) \, (1+\lambda + C_A)^2}{\sigma \Vol_g(B_g(x, r))} \|u\|^2_{L^2(\Omega_{R_\lambda};\, dv)} (R+2)^{2q}.
\]
The result then follows from applying Lemma \ref{lem:vollower}.
\end{proof}

\begin{proof}[Proof of Theorem~\ref{thm:polynomial-growth}]
The theorem follows from Corollary \ref{cor:polynomialgrowth} and \eqref{eq:growthf}.
\end{proof}

% =====================================================
\section{Spectral counting function estimate with sharp order}
\label{sec:spectralcounting}
In this section we derive an upper bound for the counting function of the operator $L$ in the setting of section \ref{sec:agmon}. We restate Theorem~\ref{thm:intro-counting} as:
\begin{theorem}\label{thm:sharp}
Let $(M^n,g,f)$ be a complete gradient Ricci shrinker.
Let $E\to M$ be a Hermitian vector bundle with compatible connection, and let $L=-\Delta+\D_{\D f}+A$ where $A$ satisfies
$\langle A(u),u\rangle \ge -C_A |u|^2$ for some $C_A\ge 0$.  Assume:
\begin{enumerate}
\item[(H2)] There exists $\beta \in [1,n]$, such that for all
$\lambda > 0$,
\[
\Vol_g(\Omega_\lambda) \le C(n) \,\lambda^{\,\beta/2},
\qquad \Omega_\lambda := \{x\in M : f(x) \le \lambda\}.
\]
\end{enumerate}
Then the spectral counting function satisfies
\begin{equation}\label{eq:mainbound}
\mathcal{N}_L(\lambda)\le C(n)\, \rank(E)\,(1+\lambda+C_A)^\frac{n+\beta}{2},\qquad\lambda\ge0.
\end{equation}
\end{theorem}

The remainder of this section is devoted to the proof.
Fix $\lambda\ge0$ and set
\begin{equation}\label{eq:adef}
a:=2(\lambda+C_A + 1).
\end{equation}
Let $E_\lambda:=\operatorname{span}\{u_i:\lambda_i\le\lambda\}$ where $u_i$ are the eigensections satisfying $Lu_i = \lambda_i u_i$, $i = 1,2,...$; then $\dim E_\lambda=\mathcal{N}_L(\lambda)$.
For $u\in E_\lambda$ write $\tilde u:=e^{-f/2}u$ for the conjugated section, which satisfies \eqref{eq:eigentransformed}. Since $L$ and $\tilde L$ are isospectral, we will estimate $\mathcal{N}_{\tilde L}(\lambda)$ instead.

% ============================================================
\medskip\noindent
\textbf{Localization to classically allowed region.}
For each eigensection $L u_i = \lambda_i u_i$, it is straight forward to check that $u_i \in W^{1,2}(d\mu)$. Hence the Agmon estimate in Lemma~\ref{lem:agmon-shrinker} applies to $u\in E_\lambda$. We now convert the estimate into an $L^2(dv)$ localization for $\tilde u$.

\begin{lemma}[$L^2$-Localization]\label{lem:localization}
For any $\varepsilon\in(0,1)$ there exists $R=C(n,\varepsilon)a$ such that for every $\tilde u\in \tilde E_\lambda$,
\begin{equation}\label{eq:truncation}
\int_{\{f\ge R\}}|\tilde u|^2\,dv \le\varepsilon\int_M|\tilde u|^2\,dv.
\end{equation}
\end{lemma}

\begin{proof}
From Lemma~\ref{lem:agmon-shrinker} with $\epsilon = \tfrac14$,
$b=\frac n4+C_A+\lambda$, $R_\lambda=8(b+1)$, we have
\[
\int_{\{f>R_\lambda\}} e^{ f/2 } |\tilde u|^2\,dv
\le R_\lambda e^{R_\lambda/2} \int_{\{f\le R_\lambda\}}|\tilde u|^2\,dv.
\]
Normalize $\|\tilde u\|_{L^2(dv)} = \|u\|_{L^2(d\mu)}=1$, so the right-hand side is at most $R_\lambda e^{R_\lambda/2}$.
For any $R\ge R_\lambda$,
\[
\int_{\{f>R\}}|\tilde u|^2\,dv \leq e^{-R/2} R_\lambda e^{R_\lambda /2}.
\]
It suffices to take $R \ge 2 (R_\lambda - \ln \varepsilon)$ to make the RHS above less than $\varepsilon$. We can take $R = \max\{8n , - \ln \varepsilon \} (1+\lambda + C_A)$.
The conclusion then follows.
\end{proof}

For the rest of the proof we fix $\varepsilon=1/4$ and let $R=C'(n)a$ be the corresponding truncation radius from Lemma~\ref{lem:localization}.
Define $\Omega:=\{f<R\}$.  From~\eqref{eq:truncation} we obtain
\begin{equation}\label{eq:tilde-trunc}
\|\tilde u\|_{L^2(\Omega,dv)}^2\ge\frac34\|\tilde u\|_{L^2(M,dv)}^2.
\end{equation}
Hence the restriction map
\[
\mathcal{R}: E_\lambda\longrightarrow L^2(\Omega,dv,E),\qquad
\mathcal{R}(u)=\tilde u|_\Omega
\]
is injective.  Set $\widetilde E_\lambda:=\mathcal{R}(E_\lambda)$; then $\dim\widetilde E_\lambda=\mathcal N_{\tilde L} (\lambda)$.

% ============================================================
\medskip\noindent
\textbf{Covering of the allowed region}
Choose a small constant $\rho_0=\rho_0(n)>0$ such that $\rho_0\le\frac14 e^{-2/(n-2)}$ and the Sobolev inequality
of Proposition~\ref{prop:Sobolev} holds on all balls $B_g(q,r)$ with $q\in\Omega_{ka}$, $k\ge1$, $a \ge1$,
and $r\le\rho_0/\sqrt{ka}$.
Set
\begin{equation}\label{eq:radius}
r_0:=\frac{\rho_0}{\sqrt{R}},\qquad r_1:=\frac{r_0}{2}.
\end{equation}
Recall that $R= C' a$ for some constant $C'$ depending only on $n$, without loss of generality we can take $C' \ge 1$, so $r_0\le \frac{1}{4} e^{-2/(n-2)}/\sqrt{C'a}$ and all local estimates apply on balls of radius $r_0$ centered in $\Omega = \Omega_R$.

Choose a maximal set of disjoint balls $\{B_g(x_j,r_1)\}_{j=1}^N$ with centers $x_j\in\Omega$.  By maximality, the enlarged balls
$\hat B_j:=B_g(x_j,r_0)$ cover $\Omega$. Each $B_g(x_j,r_1)$ is contained in $\{f\le R+2\}$ by Lemma~\ref{lem:f-control}.  Using the assumption $(H2)$,
\[
\Vol_g(\{f\le R+2\})\le C(n)R^{\beta/2}.
\]
Together with Lemma~\ref{lem:vollower}, we have
\[
N\, v\,r_1^{\,n}
\le\sum_{j=1}^N\Vol_g(B_g(x_j,r_1))
\le\Vol_g(\{f\le R+2\})\le C(n)R^{\beta/2}.
\]
Since $r_1=\rho_0/(2\sqrt{R})$ and $R=C' a$, we obtain
\begin{equation}\label{eq:coverNumber}
N\le C(n)\,\frac{R^{\beta/2}}{r_1^{\,n}}
\le C(n)\,(1+\lambda+C_A)^\frac{n+\beta}{2}.
\end{equation}

% ============================================================
\medskip\noindent
\textbf{Partition of unity and direct sum embedding.}
Now we implement the IMS localization (\cite[Chapter 3]{CFKS}) in the form of a direct sum embedding with control of the energy.

The local doubling estimate in Lemma~\ref{lem:doubling} implies that the cover $\{\hat B_j\}$ has bounded overlap: each point of $\Omega$ lies in
at most $K=K(n)$ of the balls $\hat B_j$. By a standard construction, there exists a smooth partition of unity $\{\chi_j\}$ on $\Omega$ such that
\[
0\le\chi_j\le1,\quad
\supp(\chi_j)\subset\hat B_j,\quad
\sum_{j=1}^N\chi_j\equiv1\ \text{on}\ \Omega,\quad
|\nabla\chi_j|\le\frac{C_\chi}{r_0},
\]
with $C_\chi=C_\chi(n,K)$.  Set $\eta_j:=\sqrt{\chi_j}$; then
$\sum_{j=1}^N\eta_j^2\equiv1$ on $\Omega$,
$\supp(\eta_j)\subset\hat B_j$, and
$|\nabla\eta_j|\le C_\eta/r_0$ with $C_\eta=C_\eta(n)$.

Since $\supp(\eta_j)\subset\hat B_j$, for any $\tilde v\in\widetilde E_\lambda$ the product $\eta_j\tilde v$ is supported in the interior of $\hat B_j$ and
therefore belongs to $H^1_0(\hat B_j,dv,E)$.
Define the linear map
\[
\Phi:\widetilde E_\lambda\longrightarrow
\bigoplus_{j=1}^N H^1_0(\hat B_j,dv,E),\qquad
\Phi(\tilde v)=(\eta_1\tilde v,\dots,\eta_N\tilde v).
\]
The norm on the direct sum space is defined as
\[
\|(\tilde w_1, \tilde w_2, ..., \tilde w_N)\|_\oplus^2
:=\sum_{j=1}^N\int_{\hat B_j}|\tilde w_j|^2\,dv.
\]
\begin{lemma}[Isometric embedding]\label{lem:isometry}
For all $\tilde v\in\widetilde E_\lambda$, we have $\|\Phi(\tilde v)\|_\oplus^2 =\int_\Omega|\tilde v|^2\,dv$.
Consequently $\Phi$ is injective.
\end{lemma}

\begin{proof}
$\sum\eta_j^2\equiv1$ on $\Omega$ gives the pointwise identity $|\tilde v|^2=\sum_j\eta_j^2|\tilde v|^2$; integration yields the
norm equality.  Injectivity follows naturally.
\end{proof}

Let $\mathcal{E}_\oplus$ be the Dirichlet form on the direct sum:
\[
\mathcal{E}_\oplus((w_1,\dots,w_N))
:=\sum_{j=1}^N\int_{\hat B_j}|\nabla w_j|^2\,dv.
\]
The error term in the IMS localization caused by derivatives of the partition of unity is handled by the following energy estimate. 
\begin{lemma}[Energy bound]\label{lem:energy}
There is a constant $C_E=C_E(n)$, such that for every $\tilde v\in\widetilde E_\lambda$, 
\[
\mathcal{E}_\oplus(\Phi(\tilde v))
\le C_E\,a\,\|\Phi(\tilde v)\|_\oplus^2.
\]

\end{lemma}

\begin{proof}
Let $\tilde v=\tilde u|_\Omega$ with $\tilde u\in \tilde E_\lambda$.
For each $j$,
\[
\int_{\hat B_j}|\nabla(\eta_j\tilde v)|^2\,dv
\le2\int_{\hat B_j}\eta_j^2|\nabla\tilde v|^2\,dv
+2\int_{\hat B_j}|\nabla\eta_j|^2|\tilde v|^2\,dv.
\]
Summing over $j$ and using $\sum\eta_j^2=1$ on $\Omega$,
\begin{equation}\label{eq:energySplit}
\mathcal{E}_\oplus(\Phi(\tilde v))
\le2\int_\Omega|\nabla\tilde v|^2\,dv
+2\frac{C_\eta^2}{r_0^2}\sum_{j=1}^N\int_{\supp(\eta_j)}|\tilde v|^2\,dv.
\end{equation}

For the first term on the RHS, use the eigenvalue equation.  Since $\tilde v$ is the restriction of $\tilde u$ on $\Omega$,  $\tilde u$
satisfies \eqref{eq:eigentransformed} and $V_f = f/4+ \Sc/4-n/4\ge-n/4$,
\[
\begin{split}
\int_\Omega|\nabla\tilde v|^2\,dv
\le\int_M|\nabla\tilde u|^2\,dv
=& \langle-\Delta\tilde u,\tilde u\rangle_{L^2(M,dv)} = \langle (\tilde L -V_f - A) \tilde u,\tilde u\rangle_{L^2(M,dv)} \\
\le &\bigl(\lambda+\tfrac n4+C_A\bigr)\int_M|\tilde u|^2\,dv
\le(\lambda+\tfrac n4+C_A)\cdot\frac43\|\tilde v\|_{L^2(\Omega,dv)}^2,
\end{split}
\]
where we used \eqref{eq:tilde-trunc} in the last step.
Since $a=2(\lambda+C_A+1)$, we have
$\lambda+\frac n4+C_A\le C(n)\,a$.

For the second term on the RHS of~\eqref{eq:energySplit}, since each $x\in\Omega$ belongs to at most $K$ of the $\hat B_j$, we have
\[
\sum_{j=1}^N\int_{\supp(\eta_j)}|\tilde v|^2\,dv
\le K\int_\Omega|\tilde v|^2\,dv
=K\|\Phi(\tilde v)\|_\oplus^2.
\]
With $r_0^{-2}=C'a/\rho_0^2$, this term is bounded by
$2C_\eta^2 K\rho_0^{-2} C'\,a\,\|\Phi(\tilde v)\|_\oplus^2$.

Combining the two estimates yields the result.
\end{proof}

% ============================================================
\medskip\noindent
\textbf{Local counting function estimate.}
Following the method of Cheng-Li \cite{ChengLi}, we prove the following:

\begin{lemma}[Local counting function]\label{lem:localCounting}
Let $\hat B=\hat B_j$ be one of the covering balls, with
$\hat B=B_g(q,r_0)$, $r_0$ as in \eqref{eq:radius}. Let $\tilde L_{D, \hat B}$ be the restriction of $\tilde L$ to $H^1_0(\hat B, dv, E)$. Let $\lambda_{D, j}$, $j = 1,2,...$, be the eigenvalues of the Dirichlet problem 
\[
\tilde L_{D, \hat B} \phi_j = \lambda_{D, j} \phi_j \ \text{on} \ \hat B, \quad \phi_j = 0 \ \text{on} \ \partial \hat B,
\]
where $\phi_j \in H^1_0(\hat B, dv, E)$ are the corresponding eigensections with $\int_{\hat B} |\phi|^2_j dv = 1$. 
Let $\mathcal N_{\tilde L_{D, \hat B}} (\lambda) = \# \{\lambda_{D, j} : \lambda_{D, j} \le \lambda \}$ be the corresponding counting function. Then, for any $k \geq 1$,
\[
\mathcal N_{\tilde L_{D, \hat B}}(ka ) \le C(n) k^{n/2}\rank(E).
\]
\end{lemma}

\begin{proof}
Let $H(x, y ,t)$ be the symmetric Dirichlet heat kernel of the operator $\tilde L$ on $\hat B$, then
\[
H(x, y,t) = \sum_{j = 1}^\infty e^{-\lambda_{D, j} t} \phi_j(x) \otimes \phi^*_j(y),
\]
where $\phi_j^*$ is the dual section defined by $\phi_j^*(u) = \langle u, \phi_j\rangle$.
Define the trace 
\[
\tr_E H (x,x,t) = \sum_{j = 1}^\infty e^{-\lambda_j t} |\phi_j(x)|^2.
\]
Observe that $\int_{\hat B} \tr_E H(x, x, \tfrac{1}{\lambda}) dv \geq \mathcal N_{\tilde L_{D, \hat B} }(\lambda) e^{-1} $. Hence we only need to estimate the trace $\tr_E H(x, x, 2t)$ from above, then plug in $t = \tfrac{1}{2\lambda}$. 

By the semigroup property,
\[
\tr_E H(x, x, 2t) = \tr_E \int_{\hat B}  H(x, \cdot, t) H(\cdot, x, t ) dv =  \int_{\hat B} |H(x, \cdot, t)|^2 dv .
\]
Use H\"older inequality with $p = \tfrac{n+2}{n-2}$, we have 
\[
\int_{\hat B} |H(x, \cdot, t)|^2 dv \le  \left( \int_{\hat B} |H(x, \cdot, t)|^{\frac{2n}{n-2}} dv \right)^{\frac{n-2}{n+2} } \left( \int_{\hat B} |H(x, \cdot , t)| dv \right)^{\frac{4}{n+2}}.
\]
By Kato's inequality, we have the distributional inequality
$|u|\Delta |u| \ge \langle u, \Delta u\rangle$ for any smooth section $u$; together with the lower bound $\langle (V_f+A)u,u\rangle \ge (V_f-C_A)|u|^2$, this yields the requirement of the abstract domination theorem of Hess--Schrader--Uhlenbrock~\cite[Theorem~2.15]{HUS}.  Consequently the Dirichlet heat semigroup $e^{-t\tilde L_{D,\hat B}}$ on $\hat B$ is dominated by the scalar semigroup $e^{-t(-\Delta+V_f-C_A)_{D,\hat B}}$; i.e.
$|e^{-t\tilde L} u| \le e^{-t(-\Delta+V_f-C_A)}|u|$.  Passing to heat kernels
gives the pointwise operator-norm estimate $\|H(x,y,t)\|_{\mathrm{op}} \le p(x,y,t)$, where $p$ is the scalar Dirichlet heat kernel for $-\Delta + (V_f)_+ + C_A$.  Hence for the Hilbert--Schmidt norm we obtain
\[
|H(x,y,t)| \le \sqrt{\rank(E)}\, p(x,y,t).
\]
(One may also give an alternative proof using the differential inequality $\partial_t |H| \le \Delta|H| - (V_f)_+ |H| + C_A|H|$, obtained from Kato's
inequality, together with the maximum principle.)

Since
\[
\frac{d}{d t} \int_{\hat B} p(x, \cdot, t) dv \le \int_{\hat B} \Delta p(x, \cdot, t) dv + C_A \int_{\hat B} p(x, \cdot, t) dv,
\]
we get $\int_{\hat B} p(x, \cdot, t) dv \le e^{C_A t}$.
 Consequently, 
\[
\int_{\hat B} |H(x, \cdot, t)| dv \leq \sqrt{\rank(E)} e^{C_A t}.
\]
By the local Sobolev inequality Proposition \ref{prop:Sobolev} and Kato's inequality, 
\[
\left( \int_{\hat B} |H(x, \cdot, t)|^{\frac{2n}{n-2}} dv \right)^{\frac{n-2}{n} } \le \frac{C_S(n) r_0^2}{\Vol_g(\hat B)^{2/n}}\int_{\hat B} |\D H(x, \cdot, t)|^2 dv.
\]
Integration by parts yields
\[
\begin{split}
\int_{\hat B} |\D H|^2 dv = -\int_{\hat B}\langle H , \Delta H \rangle dv  =&  - \frac{1}{2} \frac{\pd}{\pd t}\int_{\hat B}| H|^2 dv - \int_{\hat B} \langle V_f H + A(H), H\rangle dv \\
\le &  - \frac{1}{2} \frac{\pd}{\pd t}\int_{\hat B}| H|^2 dv + ( n/4+C_A) \int_{\hat B} |H|^2 dv,
\end{split}
\]
where we used the lower bound $V_f \geq - \tfrac n 4$ and the assumption $\langle A(u), u\rangle \ge - C_A |u|^2$.

Denote $h = \tr_E H(x,x,2t) $. Combining the above inequalities yields
\[
\frac{\pd}{\pd t} h^{-\frac{2}{n}} \geq - \frac 4 n \left( \frac n 4 + C_A\right) h^{-\frac{2}{n}} + \frac{4}{n} \frac{\Vol_g(\hat B)}{C_S r_0^2} e^{-\frac{4}{n} C_A t} \rank(E)^{- \frac 2n}.
\]
Note that $\lim_{t \to 0} h^{-2/n} = 0$. Solving this differential inequality yields
\[
h \le \left( \frac n 4 C_S \right)^{\frac n 2} (e^t - 1)^{-\frac n 2} e^{ (n/2 + 2 C_A) t } \frac{r_0^n}{\Vol_g(\hat B)} \rank(E).
\]
Plug in $t = \tfrac{1}{2ka}$ and integrate on $\hat B$ yields 
\[
\mathcal N_{\tilde L_{D, \hat B}}(ka ) \leq e \int_{\hat B} h dv \leq C(n) (k a)^{n/2} r_0^n \rank(E).
\]
Since $r_0=\rho_0(n)/\sqrt{C'a}$, the estimate follows. 
\end{proof}

% ============================================================
\medskip\noindent
\textbf{Completion of the proof of counting function estimates.}

\begin{proof}[Proof of Theorem~\ref{thm:sharp}]
By Lemma~\ref{lem:isometry}, $\Phi(\widetilde E_\lambda)$ is an $\mathcal N_{\tilde L}(\lambda)$-dimensional subspace of the direct sum space
$\bigoplus_{j=1}^N H^1_0(\hat B_j,dv,E)$.
By Lemma~\ref{lem:energy}, every $w\in\Phi(\widetilde E_\lambda)$ satisfies
\[
\mathcal{E}_\oplus(w)\le C_E a\,\|w\|_\oplus^2.
\]
The min-max principle implies that if an $m$-dimensional subspace $W$ satisfies $\mathcal{E}_\oplus(w) \le K \, \|w\|_\oplus^2 $ for all $w \in W$, then the operator $\bigoplus_{j =1}^N \tilde L_{D, \hat B_j}$ has at least $m$ eigenvalues not exceeding $K$. Hence 
\[
\mathcal N_{\tilde L}(\lambda)=\dim\Phi(\widetilde E_\lambda)
\le\mathcal{N}_\oplus(C_E a),
\]
where $\mathcal{N}_\oplus(\lambda)$ denotes the number of eigenvalues of the direct sum operator $\bigoplus_{j=1}^N \tilde L_{D,\hat B_j}$ not
exceeding $\lambda$, and $\tilde L_{D, \hat B_j}$ denotes the restriction of the operator $\tilde L$ to sections on $\hat B_j$ with Dirichlet boundary condition. 
It is known (\cite[page 268]{ReedSimon}) that $\mathcal{N}_\oplus(\lambda) = \sum_{j = 1}^N \mathcal N_{\tilde L_{D, \hat B_j}}( \lambda)$.

By Lemma~\ref{lem:localCounting} (with $k=C_E$), each $\hat B_j$ contributes at most $C(n) \, C_E^{n/2}\rank(E)$ eigenvalues below
$C_E a$.  Hence
\[
\mathcal{N}_\oplus(C_E a)\le N\, C(n) \, C_E^{n/2}\rank(E).
\]

Using the covering number estimate~\eqref{eq:coverNumber},
\[
\mathcal N_{\tilde L}(\lambda)\le C(n)\,(1+\lambda+C_A)^\frac{n+\beta}{2}\, C_E^{n/2}\rank(E)
=C(n)\,  \rank(E)\,(1+\lambda+C_A)^{\frac{n+\beta}{2}}.
\]
This completes the proof.
\end{proof}

\medskip\noindent
\textbf{Proofs of the corollaries.} Now we prove the corollaries of the counting number estimates. 

\begin{proof}[Proof of Corollary \ref{cor:intro-lower}]
By Theorem~\ref{thm:intro-counting}, we have 
\[
k \le C(n) (1+\lambda_k(-\Delta_f))^{\frac{n+\beta}{2}} \le C(n)' \lambda_k(-\Delta_f)^{\frac{n+\beta}{2}},
\]
where we have used the fact that $\lambda_1(-\Delta_f) \geq \frac{1}{2}$ ~\cite{HeinNaber}. The result then follows. 
\end{proof}

To prove Corollary \ref{cor:intro-betti}, we first introduce some notations. Let $d\mu = e^{-f} dv$, and $\delta_f= \delta + \iota_{\D f}$ be the $L^2(d\mu)$-adjoint of the exterior differential operator $\dd$. The $f$-Hodge Laplacian operator is 
\[
\Delta^{\dd}_f = \dd \delta_f + \delta_f \dd.
\]
By the Weitzenb\"ock formula it can be written as 
\[
\Delta^{\dd}_f \omega = -\Delta \omega + \D_{\D f} \omega + \mathcal{R}_f(\omega). 
\]
See \cite[Lemma 2.2]{HeL2} for an explicit formula for $\mathcal{R}_f(\omega)$ on gradient Ricci shrinkers. A form $\omega$ is called $f$-harmonic if $\Delta_f^{\dd} \omega = 0$. Let $\mathcal{H}^p_f(M)$ be the space of $f$-harmonic $p$-forms on $M$ that are $L^2(d\mu)$-integrable. 

\begin{proof}[Proof of Corollary \ref{cor:intro-betti}]
By \cite[Theorem 1.4]{HeL2}, $\dim \mathcal{H}^p_f(M) = b_p(M)$ under the assumption that $|\Ric| \le \tfrac{1}{5n} f$ when $f$ is sufficiently large. Hence we have $b_p(M) \le \mathcal{N}_{\Delta^{\dd}_f}(0)$, then the estimate follows from Theorem~\ref{thm:intro-counting}.
\end{proof}

 %=====================================================
% Bibliography
% =====================================================
\bibliographystyle{plain}

%\vspace*{\fill}
\contactinfo

\end{document}